\documentclass{amsproc}
\usepackage{breqn,float,amssymb,pdflscape,rotating,multirow,hyperref,xcolor,setspace,graphicx}
\makeatletter
\@namedef{subjclassname@2020}{%
  \textup{2020} Mathematics Subject Classification}
\makeatother
\newtheorem{theorem}{Theorem}[section]

\theoremstyle{definition}

\newtheorem{example}[theorem]{Example}

\theoremstyle{remark}

\numberwithin{equation}{section}
\allowdisplaybreaks
\begin{document}
\onehalfspacing
%
\title[Mittag-Leffler function]{Double and single integrals of the Mittag-Leffler Function: Derivation and Evaluation}


\author{Robert Reynolds}
\address[Robert Reynolds]{Department of Mathematics and Statistics, York University, Toronto, ON, Canada, M3J1P3}
\email[Corresponding author]{milver@my.yorku.ca}
\thanks{}


\date{}

\dedicatory{}

\begin{abstract}
One-dimensional and two-dimensional integrals containing $E_b(-u)$ and $E_{\alpha ,\beta }\left(\delta  x^{\gamma }\right)$ are considered. $E_b(-u)$ is the Mittag-Leffler function and the integral is taken over the rectangle $0 \leq x < \infty, 0 \leq u < \infty$ and $E_{\alpha ,\beta }\left(\delta  x^{\gamma }\right)$ is the generalized Mittag-Leffler function and the integral is over $0\leq x \leq b$ with infinite intervals explored. A representation in terms of the Hurwitz-Lerch zeta function and other special functions are derived for the double and single integrals, from which special cases can be evaluated in terms of special function and fundamental constants.
\end{abstract}

\maketitle
\section{Significance Statement}
The double and single definite integral of the Mittag-Leffler function expressed in terms of the Hurwitz-Lerch zeta and other special functions is the subject of this paper. Evaluation of these integrals are required in many practical areas of science and engineering, and the purpose of this concise compendium is to give the researcher the fundamental knowledge in this area. There exists a considerable body of literature on the subject of Mittag-Leffler function.

\section{Introduction}
The Mittag-Leffler function is named after Magnus Gustaf (G\"{o}sta) Mittag-Leffler (1846-1927) a Swedish mathematician. He is accredited for his mathematical contributions in the theory of functions, which is known as complex analysis. Related topics and applications of the Mittag-Leffler function are detailed in the book by Gorenflo et al. \cite{Gorenflo}. Definite integrals whose kernels contain the Mittag-Leffler function are found in the book of Gorenflo et al. \cite{Gorenflo}, work by Gubara et al. \cite{Gubara}, Haubold \cite{Haubold}, and \"{O}zarslan \cite{ozarslan} to name a few. 
The generalised hypergeometric series or generalised hypergeometric functions are a specific example of transcendents that includes the Mittag-Leffler function. Moreover, many integrals involving the Mittag-Leffler function are either bypergeometric functions themselves or may be defined in terms of them. 
A valuable feature of this paper is the treatment of composite functions of the Hurwitz-Lerch Zeta function. With these results in hand, many useful representations of Hurwitz-Lerch zeta functions and their integrals simultaneously follow parameter specialization. In addition, these results are useful for determining the properties of  Hurwitz-Lerch zeta functions and their integral representations. A short table containing the Mittag-Leffler function is also available. This table adds a lot of utility because many mathematical functions can be expressed as a series of Hurwitz-Lerch Zeta functions.
In this current work the authors expand upon current single integrals and investigate an invariant property with respect to the index $b$ of the double integral involving the product of logarithmic and Mittag-Leffler functions. This invariant property is new to the best of our knowledge particularly in the context of double integrals. A closed form solution in terms of the Hurwtiz-Lerch zeta function is derived to represent this double integral allowing for special cases to be derived in terms of other special functions. In this paper we derive the double definite integral given by: 
\begin{multline}
\int_{0}^{\infty}\int_{0}^{\infty}e^{-c x} u^{m-1} x^{-b m} E_b(-u) \log ^k\left(a u x^{-b}\right)dudx
\end{multline}
where the parameters $k,a,c$ and $m$ are general complex numbers and $Re(m)\leq 1/2, Re(b)>0, Re(c)>0$. This definite integral will be used to derive special cases in terms of special functions and fundamental constants. The derivations follow the method used by us in~\cite{reyn4}. This method involves using a form of the generalized Cauchy's integral formula given by:
\begin{equation}\label{intro:cauchy}
\frac{y^k}{\Gamma(k+1)}=\frac{1}{2\pi i}\int_{C}\frac{e^{wy}}{w^{k+1}}dw.
\end{equation}
where $C$ is, in general, an open contour in the complex plane where the bilinear concomitant has the same value at the end points of the contour. We then multiply both sides by a function of $x$ and $u$, then take a definite double integral of both sides. This yields a definite integral in terms of a contour integral. Then, we multiply both sides of Equation~(\ref{intro:cauchy})  by another function of $x$ and $u$ and take the infinite sum of both sides such that the contour integral of both equations are the same.
\section{Definite Integral of the Contour Integral}
We used the method in~\cite{reyn4}. The variable of integration in the contour integral is $r =  w+m$. The cut and contour are in the first quadrant of the complex $r$-plane.  The cut approaches the origin from the interior of the first quadrant; the contour goes round the origin with zero radius and is on opposite sides of the cut.  Using a generalization of Cauchy's integral formula, we form the double integral by replacing $y$ with 
\begin{equation*}
\log \left(a u x^{-b}\right)
\end{equation*}
and multiplying by 
\begin{equation*}
e^{-c x} u^{m-1} x^{-b m} E_b(-u)
\end{equation*}
then taking the definite integral with respect to $x\in [0,\infty)$ and $u\in [0,\infty)$ to obtain
\begin{multline}\label{dici}
\frac{1}{\Gamma(k+1)}\int_{0}^{\infty}\int_{0}^{\infty}e^{-c x} u^{m-1} x^{-b m} E_b(-u) \log ^k\left(a u x^{-b}\right)dudx\\
=\frac{1}{2\pi i}\int_{0}^{\infty}\int_{0}^{\infty}\int_{C}a^w e^{-c x} w^{-k-1} E_b(-u) u^{m+w-1} x^{-b (m+w)}dwdudx\\
=\frac{1}{2\pi i}\int_{C}\int_{0}^{\infty}\int_{0}^{\infty}a^w e^{-c x} w^{-k-1} E_b(-u) u^{m+w-1} x^{-b (m+w)}dudxdw\\
=\frac{1}{2\pi i}\int_{C}\pi  a^w w^{-k-1} \csc (\pi  (m+w)) c^{b (m+w)-1}dw
\end{multline}
from equation (19) in \cite{mainardi} and equation (3.326.2) in \cite{grad} where $0<Re(\pi  (m+w))<1, Re(b)>0$ and using the reflection formula (8.334.3) in \cite{grad} for the Gamma function. We are able to switch the order of integration over $x$, $u$ and $r$ using Fubini's theorem for multiple integrals see page 178 in \cite{gelca}, as the integrand is of bounded measure over the space $\mathbb{C} \times [0,\infty) \times [0,\infty) $.
\section{The Hurwitz--Lerch Zeta Function and Infinite Sum of the Contour Integral}

In this section we use Equation~(\ref{intro:cauchy}) to derive the contour integral representations for the Hurwitz--Lerch Zeta  function.

\subsection{The Hurwitz--Lerch Zeta Function}

The Hurwitz--Lerch Zeta function (25.14) in \cite{dlmf} has a series representation given~by
\vspace{6pt}
\begin{equation}\label{lerch:eq}
\Phi(z,s,v)=\sum_{n=0}^{\infty}(v+n)^{-s}z^{n}
\end{equation}
where $|z|<1, v \neq 0,-1,..$ and is continued analytically by its integral representation given~by
\begin{equation}\label{armenia:eq8}
\Phi(z,s,v)=\frac{1}{\Gamma(s)}\int_{0}^{\infty}\frac{t^{s-1}e^{-vt}}{1-ze^{-t}}dt=\frac{1}{\Gamma(s)}\int_{0}^{\infty}\frac{t^{s-1}e^{-(v-1)t}}{e^{t}-z}dt
\end{equation}
where $Re(v)>0$, and either $|z| \leq 1, z \neq 1, Re(s)>0$, or $z=1, Re(s)>1$. 
\subsection{Infinite Sum of the Contour Integral}
Using Equation (\ref{intro:cauchy}) and replacing $y$ with 
\begin{equation*}
\log (a)+b \log (c)+i \pi  (2 y+1)
\end{equation*}
then multiplying both sides by 
\begin{equation*}
-2 i \pi  e^{i \pi  m (2 y+1)} c^{b m-1}
\end{equation*}
then taking the infinite sum over $y\in[0,\infty)$, and simplifying in terms of the Hurwitz--Lerch Zeta function we obtain
\begin{dmath}\label{isci}
-\frac{(2 i \pi )^{k+1} e^{i \pi  m} c^{b m-1} \Phi \left(e^{2 i m
   \pi },-k,-\frac{i (\log (a)+b \log (c)+i \pi )}{2 \pi }\right)}{\Gamma(k+1)}\\
   =-\frac{1}{2\pi i}\sum_{y=0}^{\infty}\int_{C}2 i \pi  a^w w^{-k-1} e^{i \pi  (2 y+1)
   (m+w)} c^{b (m+w)-1}dw\\
   =-\frac{1}{2\pi i}\int_{C}2 i \pi  a^w w^{-k-1} \cdot \sum_{y=0}^{\infty}e^{i \pi  (2 y+1)
   (m+w)} c^{b (m+w)-1}dw\\
   =\frac{1}{2\pi i}\int_{C}\pi 
   a^w w^{-k-1} \csc (\pi  (m+w)) c^{b (m+w)-1}dw
\end{dmath}
from Equation (1.232.2) in \cite{grad} where $Im  (m+w)>0$ in order for the sum to converge.
\section{Definite Integral in Terms of the Hurwitz--Lerch Zeta Function}
\begin{theorem}
For all $Re(m)\leq 1/2, Re(b)>0,Re(c)>0$ then,
\begin{multline}\label{dilf}
\int_{0}^{\infty}\int_{0}^{\infty}e^{-c x} u^{m-1} x^{-b m} E_b(-u) \log ^k\left(a u x^{-b}\right)dudx\\
=(2
   i \pi )^{k+1} e^{i \pi  m} \left(-c^{b m-1}\right) \Phi \left(e^{2 i m
   \pi },-k,-\frac{i (\log (a)+b \log (c)+i \pi )}{2 \pi }\right)
\end{multline}
\end{theorem}
\begin{proof}
The right-hand sides of relations (\ref{dici}) and (\ref{isci}) are identical; hence, the left-hand sides of the same are identical too. Simplifying with the Gamma function yields the desired conclusion.
\end{proof}
\begin{example}
The degenerate case.
\begin{equation}
\int_{0}^{\infty}\int_{0}^{\infty}e^{-c x} u^{m-1} x^{-b m} E_b(-u)dudx=\pi  \csc (\pi  m) c^{b
   m-1}
\end{equation}
\end{example}
\begin{proof}
Use Equation (\ref{dilf}) and set $k=0$ and simplify using entry (2) in Table below (64:12:7) in \cite{atlas}.
\end{proof}
\begin{example}
The Hurwitz Zeta function $\zeta(s,v)$
\begin{multline}\label{eq:hurwitz}
\int_{0}^{\infty}\int_{0}^{\infty}\frac{e^{-x} E_{\frac{1}{4}}(-u) \log ^k\left(\frac{i
   u}{\sqrt[4]{x}}\right)}{\sqrt{u} \sqrt[8]{x}}dudx\\
   =i^k 2^{2 k+1} \pi ^{k+1}
   \left(\zeta \left(-k,\frac{3}{8}\right)-\zeta
   \left(-k,\frac{7}{8}\right)\right)
\end{multline}
\end{example}
\begin{proof}
Use Equation (\ref{dilf}) and set $m=1/2,a=i,b=1/4$ and simplify using entry (4) in Table below (64:12:70) in \cite{atlas}.
\end{proof}
\begin{example}
The digamma function $\psi^{(0)}(z)$,
\begin{equation}
\int_{0}^{\infty}\int_{0}^{\infty}\frac{e^{-x} E_{\frac{1}{4}}(-u)}{\sqrt{u} \sqrt[8]{x} \log
   \left(\frac{i u}{\sqrt[4]{x}}\right)}dudx=\frac{1}{2} i \left(\psi
   ^{(0)}\left(\frac{3}{8}\right)-\psi
   ^{(0)}\left(\frac{7}{8}\right)\right)
\end{equation}
\end{example}
\begin{proof}
Use Equation (\ref{eq:hurwitz}) and apply l'Hopital's rule as $k\to -1$ and simplify using equation (64:4:1) in \cite{atlas}.
\end{proof}
\begin{example}
\begin{multline}
\int_{0}^{\infty}\int_{0}^{\infty}\frac{e^{-x} E_{\beta }(-u) x^{-(\beta  (m+n))} \left(u^n x^{\beta 
   m}-u^m x^{\beta  n}\right)}{u \log \left(u x^{-\beta }\right)}dudx\\
   =\log
   \left(\cot \left(\frac{\pi  m}{2}\right) \tan \left(\frac{\pi 
   n}{2}\right)\right)
\end{multline}
\end{example}
\begin{proof}
Use Equation (\ref{dilf}) and form a second equation by replacing $m\to n$ and taking their difference. Next set $k=-1,a=1,c=1$ and simplify using entry (3) in table below (64:12:7) in \cite{atlas}.
\end{proof}
\section{The invariance of index $b$ relative to the Hurwitz--Lerch Zeta Function}
In this section, we evaluate Equation (\ref{dilf}) such that the index of the Mittag-Leffler function $b$ is independent of the right-hand side. This type of integral could involve properties related to orthogonal functions. This invariant property occurs as a result of how the gamma function is chosen for the definite integral of the contour integral to reduce to a trigonometric function. The derivation of this invariant property is not dependent on all the parameters involved.
\begin{example}
The Hurwitz-Lerch Zeta function $\Phi(s,u,v)$,
\begin{multline}\label{eq:phi-inv}
\int_{0}^{\infty}\int_{0}^{\infty}e^{-x} u^{m-1} x^{-b m} E_b(-u) \log ^k\left(a u x^{-b}\right)dudx\\
=(2 i
   \pi )^{k+1} \left(-e^{i \pi  m}\right) \Phi \left(e^{2 i m \pi
   },-k,-\frac{i (\log (a)+i \pi )}{2 \pi }\right)
\end{multline}
\end{example}
\begin{proof}
Use Equation (\ref{dilf}) and set $c=1$ and simplify.
\end{proof}
\begin{example}
The Polylogarithm function $Li_{k}(z)$,
\begin{multline}\label{eq:poly}
\int_{0}^{\infty}\int_{0}^{\infty}e^{-x} u^{m-1} x^{-b m} E_b(-u) \log ^k\left(-u x^{-b}\right)dudx\\
=(2 i
   \pi )^{k+1} \left(-e^{-i \pi  m}\right) \text{Li}_{-k}\left(e^{2 i m \pi
   }\right)
\end{multline}
\end{example}
\begin{proof}
Use Equation (\ref{dilf}) and set $a=-1,c=1$; simplify using Equation~(64:12:2) in \cite{atlas}.
\end{proof}
\begin{example}
The constant $\pi$,
\begin{equation}
\int_{0}^{\infty}\int_{0}^{\infty}\frac{e^{-x} x^{-b/2} E_b(-u)}{\sqrt{u} \log ^2\left(-u
   x^{-b}\right)}dudx=-\frac{\pi }{24}
\end{equation}
\end{example}
\begin{proof}
Use Equation (\ref{eq:poly}) and set $k=-2,m=1/2$ and simplify. 
\end{proof}
\begin{example}
Catalan's constant $K$,
\begin{equation}
\int_{0}^{\infty}\int_{0}^{\infty}\frac{e^{-x} x^{-b/4} E_b(-u)}{u^{3/4} \log ^2\left(-u
   x^{-b}\right)}dudx=-\frac{(-1)^{1/4} \left(\pi ^2-48 i K\right)}{96 \pi
   }
\end{equation}
\end{example}
\begin{proof}
Use Equation (\ref{eq:poly}) and set $k=-2,m=1/2$ and simplify using equation (2.2.1.2.7) in \cite{lewin}.
\end{proof}
\section{The Log-gamma Representation}
\begin{example}
\begin{multline}\label{eq:lg}
\int_{{\mathbb{R}^{2}_{+}}}\frac{x^{-b/4} e^{-c x} E_b(-u) \log \left(\log \left(a u
   x^{-b}\right)\right)}{u^{3/4}}dudx\\
   =\frac{\pi  c^{\frac{b}{4}-1}}{\sqrt{2}} \left((-2+2
   i) \log \left(\frac{\Gamma \left(-\frac{i (\log (a)+b \log (c)+i \pi
   )}{8 \pi }\right)}{\Gamma \left(-\frac{i (\log (a)+b \log (c)+5 i \pi
   )}{8 \pi }\right)}\right)\right.\\
   \left.-(2+2 i) \log \left(\frac{\Gamma \left(-\frac{i
   (\log (a)+b \log (c)+3 i \pi )}{8 \pi }\right)}{\Gamma \left(-\frac{i
   (\log (a)+b \log (c)+7 i \pi )}{8 \pi }\right)}\right)+i \pi +2 \log
   (\pi )+\log (64)\right)
\end{multline}
\end{example}
\begin{proof}
Use Equation (\ref{dilf}) and take the first partial derivative with respect to $k$ and set $k=0$ and simplify using equation (6.7) in \cite{reyn_aims}.
\end{proof}
\begin{example}
\begin{multline}\label{eq:lg1}
\int_{{\mathbb{R}^{2}_{+}}}\frac{e^{-e^{i a} x} x^{\frac{1}{4}-\frac{b}{4}} E_{b-1}(-u) \log
   \left(\log \left(u x^{1-b}\right)+i a\right)}{u^{3/4}}dudx\\
   =\frac{\pi 
   e^{\frac{1}{4} i a (b-5)}}{\sqrt{2}} \left((-2-2 i) \log \left(\frac{\Gamma
   \left(\frac{1}{8} \left(\frac{a b}{\pi }+3\right)\right)}{\Gamma
   \left(\frac{1}{8} \left(\frac{a b}{\pi }+7\right)\right)}\right)\right.\\
   \left.-(2-2 i)
   \log \left(\frac{\Gamma \left(\frac{a b+\pi }{8 \pi }\right)}{\Gamma
   \left(\frac{1}{8} \left(\frac{a b}{\pi }+5\right)\right)}\right)+i \pi +2
   \log (\pi )+\log (64)\right)
\end{multline}
\end{example}
\begin{proof}
Use equation (\ref{eq:lg}) and set $a=e^{ai},c=a,b=b-1$ and simplify.
\end{proof}
\begin{example}
\begin{multline}
\int_{{\mathbb{R}^{2}_{+}}}\frac{e^{-\sqrt[5]{-1} x} E_{\frac{1}{3}}(-u) \log \left(\log
   \left(\frac{u}{\sqrt[3]{x}}\right)+\frac{i \pi }{5}\right)}{u^{3/4}
   \sqrt[12]{x}}dudx\\
   =\frac{(-1)^{19/60} \pi }{\sqrt{2}} \left(\pi -i \log (64)-2 i \log (\pi
   )+(2+2 i) \log \left(\frac{\Gamma \left(\frac{19}{120}\right)}{\Gamma
   \left(\frac{79}{120}\right)}\right)\right.\\
   \left.-(2-2 i) \log \left(\frac{\Gamma
   \left(\frac{49}{120}\right)}{\Gamma
   \left(\frac{109}{120}\right)}\right)\right)
\end{multline}
\end{example}
\begin{proof}
Use equation (\ref{eq:lg1}) and set $a=\pi/5,b=4/3$ and simplify.
\end{proof}
\begin{example}
\begin{multline}
\int_{{\mathbb{R}^{2}_{+}}}\frac{e^{-\sqrt[3]{-1} x} E_{\frac{1}{3}}(-u) \log \left(\log
   \left(\frac{u}{\sqrt[3]{x}}\right)+\frac{i \pi }{3}\right)}{u^{3/4}
   \sqrt[12]{x}}dudx\\
   =\frac{(-1)^{7/36} \pi  \left(\pi -i \log (64)-2 i \log (\pi
   )+(2+2 i) \log \left(\frac{\Gamma \left(\frac{13}{72}\right)}{\Gamma
   \left(\frac{49}{72}\right)}\right)-(2-2 i) \log \left(\frac{\Gamma
   \left(\frac{31}{72}\right)}{\Gamma
   \left(\frac{67}{72}\right)}\right)\right)}{\sqrt{2}}
\end{multline}
\end{example}
\begin{proof}
Use equation (\ref{eq:lg1}) and set $a=\pi/3,b=4/3$ and simplify.
\end{proof}
\begin{example}
\begin{multline}
\int_{{\mathbb{R}^{2}_{+}}}\frac{e^{-\sqrt[4]{-1} x} E_{\frac{1}{4}}(-u) \log \left(\log
   \left(\frac{u}{\sqrt[4]{x}}\right)+\frac{i \pi }{4}\right)}{u^{3/4}
   \sqrt[16]{x}}dudx\\
   =\frac{(-1)^{17/64} \pi  }{\sqrt{2}}\left(\pi -i \log (64)-2 i \log (\pi
   )+(2+2 i) \log \left(\frac{\Gamma \left(\frac{21}{128}\right)}{\Gamma
   \left(\frac{85}{128}\right)}\right)\right.\\
   \left.-(2-2 i) \log \left(\frac{\Gamma
   \left(\frac{53}{128}\right)}{\Gamma
   \left(\frac{117}{128}\right)}\right)\right)
\end{multline}
\end{example}
\begin{proof}
Use equation (\ref{eq:lg1}) and set $a=\pi/4,b=5/4$ and simplify.
\end{proof}
\begin{example}
\begin{multline}
\int_{{\mathbb{R}^{2}_{+}}}\frac{e^{-(-1)^{3/7} x} E_{\frac{1}{5}}(-u) \log \left(\log
   \left(\frac{u}{\sqrt[5]{x}}\right)+\frac{3 i \pi }{7}\right)}{u^{3/4}
   \sqrt[20]{x}}dudx\\
   =\frac{(-1)^{13/140} \pi  }{\sqrt{2}}\left(\pi -i \log (64)-2 i \log (\pi
   )+(2+2 i) \log \left(\frac{\Gamma \left(\frac{53}{280}\right)}{\Gamma
   \left(\frac{193}{280}\right)}\right)\right.\\
   \left.-(2-2 i) \log \left(\frac{\Gamma
   \left(\frac{123}{280}\right)}{\Gamma
   \left(\frac{263}{280}\right)}\right)\right)
\end{multline}
\end{example}
\begin{proof}
Use equation (\ref{eq:lg1}) and set $a=3\pi/7,b=6/5$ and simplify.
\end{proof}
\begin{example}
\begin{multline}\label{eq:lg2}
\int_{{\mathbb{R}^{2}_{+}}}\frac{x^{-b/2} e^{-c x} E_b(-u) \log \left(\log \left(a u
   x^{-b}\right)\right)}{\sqrt{u}}dudx\\
   =\frac{1}{2} \pi  c^{\frac{b}{2}-1} \left(\log
   \left(\frac{16 \pi ^2 \Gamma \left(-\frac{i (\log (a)+b \log (c)+3 i \pi )}{4
   \pi }\right)^4}{\Gamma \left(-\frac{i (\log (a)+b \log (c)+i \pi )}{4 \pi
   }\right)^4}\right)+i \pi \right)
\end{multline}
\end{example}
\begin{proof}
Use equation (\ref{dilf}) and set $m=1/2$ then take the first partial derivative with respect to $k$ and set $k=0$ and simplify using equation (25.11.18) in \cite{dlmf}.
\end{proof}
\begin{example}
\begin{multline}
\int_{{\mathbb{R}^{2}_{+}}}\frac{e^{-\frac{\pi  x}{4}} E_{\frac{4}{3}}(-u) \log \left(\log
   \left(-\frac{u}{x^{4/3}}\right)\right)}{\sqrt{u} x^{2/3}}dudx\\
   =\frac{\pi ^{2/3}}{\sqrt[3]{2}}
   \left(i \pi +\log \left(\frac{16 \pi ^2 \Gamma \left(1+\frac{i \log
   \left(\frac{4}{\pi }\right)}{3 \pi }\right)^4}{\Gamma
   \left(\frac{1}{2}+\frac{i \log \left(\frac{4}{\pi }\right)}{3 \pi
   }\right)^4}\right)\right)
\end{multline}
\end{example}
\begin{proof}
Use equation (\ref{eq:lg2}) and set $a=-1,b=4/3,c=\pi/4$ and simplify.
\end{proof}
\begin{example}
\begin{multline}
\int_{{\mathbb{R}^{2}_{+}}}\frac{e^{-\frac{\pi  x}{5}} E_{\frac{5}{4}}(-u) \log \left(\log
   \left(\frac{i u}{x^{5/4}}\right)\right)}{\sqrt{u} x^{5/8}}dudx\\
   =\frac{1}{2}
   5^{3/8} \pi ^{5/8} \left(i \pi +\log \left(\frac{16 \pi ^2 \Gamma
   \left(\frac{7}{8}+\frac{5 i \log \left(\frac{5}{\pi }\right)}{16 \pi
   }\right)^4}{\Gamma \left(\frac{3}{8}+\frac{5 i \log \left(\frac{5}{\pi
   }\right)}{16 \pi }\right)^4}\right)\right)
\end{multline}
\end{example}
\begin{proof}
Use equation (\ref{eq:lg2}) and set $a=i,b=5/4,c=\pi/5$ and simplify.
\end{proof}
\begin{example}
\begin{multline}
\int_{{\mathbb{R}^{2}_{+}}}\frac{e^{-\frac{1}{7} (\pi +7 i) x} E_{\frac{6}{5}}(-u) \log \left(\log
   \left(\frac{(1+i) u}{x^{6/5}}\right)\right)}{\sqrt{u} x^{3/5}}dudx\\
   =\frac{\pi }{2 \left(\frac{1}{7} (\pi +7
   i)\right)^{2/5}}
   \left(i \pi +\log \left(\frac{16 \pi ^2 \Gamma \left(\frac{3}{4}-\frac{i
   \left(5 \log (1+i)+6 \log \left(\frac{1}{7} (7 i+\pi )\right)\right)}{20 \pi
   }\right)^4}{\Gamma \left(\frac{1}{4}-\frac{i \left(5 \log (1+i)+6 \log
   \left(\frac{1}{7} (7 i+\pi )\right)\right)}{20 \pi
   }\right)^4}\right)\right)
\end{multline}
\end{example}
\begin{proof}
Use equation (\ref{eq:lg2}) and set $a=1+i,b=6/5,c=\pi/7+i$ and simplify.
\end{proof}
\section{Definite integrals of the generalized Mittag-Leffler function}
In this section we take a different approach in deriving definite integrals using the contour integral method in \cite{reyn4}.
The contour integral representation form involving the generalized integral form is given by;
\begin{multline}\label{eq:j1}
\frac{1}{2\pi i}\int_{C}\int_{0}^{b}\frac{a^w \delta ^f \nu ^h (-\theta )^j w^{-k-1} \binom{\lambda }{h} x^{\gamma  f+h \mu +j \tau +m+w}}{j!\Gamma (f \alpha +\beta )}dxdw\\
=\frac{1}{2\pi i}\int_{C}\sum_{f=0}^{\infty}\sum_{h=0}^{\infty}\sum_{j=0}^{\infty}\frac{a^w \delta ^f \nu ^h (-\theta )^j w^{-k-1} \binom{\lambda }{h} b^{\gamma  f+h\mu +j \tau +m+w+1}}{j! \Gamma (f \alpha +\beta ) (\gamma  f+h \mu +j \tau +m+w+1)}dw
\end{multline}
where $Re(\gamma)>1,Re(\mu)>0,Re(b)>0$.
\subsection{Left-hand side contour integral representation}
Using a generalization of Cauchy's integral formula \ref{intro:cauchy} and the procedure in section (2.1), we form the definite integral given by;
\begin{multline}\label{eq:j2}
\int_{0}^{b}\frac{x^m e^{\theta  \left(-x^{\tau }\right)} \log ^k(a x) \left(\nu  x^{\mu }+1\right)^{\lambda } E_{\alpha
   ,\beta }\left(x^{\gamma } \delta \right)}{k!}dx\\
=\frac{1}{2\pi i}\int_{C}\int_{0}^{b}w^{-k-1} x^m (a x)^w e^{\theta  \left(-x^{\tau }\right)} \left(\nu
    x^{\mu }+1\right)^{\lambda } E_{\alpha ,\beta }\left(x^{\gamma } \delta \right)dxdw
    \end{multline}
where $Re(\gamma)>1,Re(\mu)>0,Re(b)>0$.
\subsection{Right-hand side contour integral}
Using a generalization of Cauchy's integral formula \ref{intro:cauchy}, and the procedure in [Reynolds, \href{https://doi.org/10.48550/arXiv.2502.14876}{Section 2.2}] we get;
\begin{multline}\label{eq:j3}
\sum_{f=0}^{\infty}\sum_{h=0}^{\infty}\sum_{j=0}^{\infty}\frac{\delta ^f (-\theta )^j \nu ^h \binom{\lambda }{h} \Gamma (1+k,-((1+m+f \gamma +h \mu +j \tau ) \log (a
   b)))}{j! k! \Gamma (f \alpha +\beta ) a^{1+m+f \gamma +h \mu +j \tau } (-1-m-f \gamma -h \mu -j \tau
   )^{k+1}}\\
=-\frac{1}{2\pi i}\int_{C}\sum_{f=0}^{\infty}\sum_{h=0}^{\infty}\sum_{j=0}^{\infty}\frac{a^w b^{1+m+w+f \gamma +h \mu +j \tau } w^{-1-k} \delta ^f (-\theta )^j \nu ^h \binom{\lambda
   }{h}}{(1+m+w+f \gamma +h \mu +j \tau ) j! \Gamma (f \alpha +\beta )}dw
\end{multline}
where $Re(\gamma)>1,Re(\mu)>0,Re(b)>0$.
from equation [Wolfram, \href{http://functions.wolfram.com/07.27.02.0002.01}{07.27.02.0002.01}] where $|Re(b)|<1$.
\section{Derivations and evaluations involving the definite integral of generalized Mittag-Leffler function}
\begin{theorem}
Generalized Mittag-Leffler integral.
\begin{multline}\label{eq:mitt}
\int_0^b e^{-x^{\tau } \theta } x^m \left(1+x^{\mu } \nu \right)^{\lambda } \log ^k\left(\frac{1}{a
   x}\right) E_{\alpha ,\beta }\left(x^{\gamma } \delta \right) \, dx\\
=\sum _{f=0}^{\infty } \sum _{h=0}^{\infty }
   \sum _{j=0}^{\infty } \frac{a^{-1-m-f \gamma -h \mu -j \tau } \delta ^f (-\theta )^j \nu ^h \binom{\lambda }{h}
   \Gamma \left(1+k,(1+m+f \gamma +h \mu +j \tau ) \log \left(\frac{1}{a b}\right)\right)}{j! \Gamma (f \alpha
   +\beta ) (1+m+f \gamma +h \mu +j \tau )^{k+1}}
\end{multline}
where $Re(\tau)>0,Re(\mu)>1$.
\end{theorem}
\begin{proof}
Since the right-hand sides of equations (\ref{eq:j2}) and (\ref{eq:j3}) are equal relative to equation (\ref{eq:j1}), we may equate the left-hand sides and simplify the gamma function to yield the stated result.
\end{proof}
\begin{example}
Here we use equation (\ref{eq:mitt}) and set $k\to 0,a\to 1$ and simplify the incomplete gamma function using equation [Wolfram, \href{http://functions.wolfram.com/06.07.03.0005.01}{01}];
\begin{multline}\label{eq:mitt_1}
\int_0^{\infty } e^{-x^{\tau } \theta } x^m \left(1+x^{\mu } \nu \right)^{\lambda } E_{\alpha ,\beta
   }\left(x^{\gamma } \delta \right) \, dx=\frac{1}{\tau  \theta ^{\frac{m+1}{\tau }}}\sum _{j=0}^{\infty } \sum _{k=0}^{\infty } \frac{\delta ^j \nu ^k
   \binom{\lambda }{k}}{\theta ^{\frac{j \gamma +k \mu }{\tau
   }}}\frac{ \Gamma \left(\frac{1+m+j \gamma +k \mu }{\tau }\right)}{ \Gamma (j \alpha +\beta )}
\end{multline}
where $Re(\theta)>0$.
\end{example}
\begin{example}
Here we use equation (\ref{eq:mitt_1}) and set $\lambda \to -1,\theta \to s,x\to t,m\to m-1,f\to j$ then take the indefinite integral with respect to $\nu$, replace $m\to \mu +m$ and simplify;
\begin{multline}\label{eq:mitt_2}
\int_0^{\infty } e^{-s t^{\tau }} t^{m-1} \log \left(1+t^{\mu } \nu \right) E_{\alpha ,\beta }\left(t^{\gamma
   } \delta \right) \, dt\\
=\frac{\nu  }{\tau  s^{\frac{m+\mu }{\tau }}}\sum _{j=0}^{\infty } \sum _{h=0}^{\infty } \frac{ \Gamma \left(\frac{m+j \gamma +\mu +h \mu }{\tau
   }\right)}{(1+h) \Gamma (j \alpha +\beta )}\left(s^{-\frac{\gamma }{\tau
   }} \delta \right)^j \left(-s^{-\frac{\mu }{\tau }} \nu \right)^h
\end{multline}
where $Re(\tau)>1,0< Re(\mu)<1$.
\end{example}
\begin{example}
Here we use equation (\ref{eq:mitt_2}) and form a second equation by $\nu \to -\nu$ and take their difference and simplify;
\begin{multline}
\int_0^{\infty } e^{-s t^{\tau }} t^{-1+m} \log \left(1-t^{2 \mu } \nu ^2\right) E_{\alpha ,\beta
   }\left(t^{\gamma } \delta \right) \, dt\\
=\sum _{j=0}^{\infty } \sum _{h=0}^{\infty } \frac{s^{-\frac{m+\mu }{\tau
   }} \left(s^{-\frac{\gamma }{\tau }} \delta \right)^j \nu  \left((-1)^h-1\right) \left(s^{-\frac{\mu }{\tau }} \nu
   \right)^h \Gamma \left(\frac{m+j \gamma +\mu +h \mu }{\tau }\right)}{(1+h) \tau  \Gamma (j \alpha +\beta
   )}
\end{multline}
where $Re(\tau)>1,0< Re(\mu)<1$.
\end{example}
\begin{example}
 Here we use equation (\ref{eq:mitt_2}) and form a second equation by $\nu \to -\nu$ and add both equations and simplify;
 \begin{multline}
\int_0^{\infty } e^{-s t^{\tau }} t^{-1+m} \tanh ^{-1}\left(t^{\mu } \nu \right) E_{\alpha ,\beta
   }\left(t^{\gamma } \delta \right) \, dt\\
=\sum _{j=0}^{\infty } \sum _{h=0}^{\infty } \frac{\left(1+(-1)^h\right)
   s^{-\frac{m+\mu }{\tau }} \left(s^{-\frac{\gamma }{\tau }} \delta \right)^j \nu  \left(s^{-\frac{\mu }{\tau }}
   \nu \right)^h \Gamma \left(\frac{m+j \gamma +\mu +h \mu }{\tau }\right)}{2 (1+h) \tau  \Gamma (j \alpha +\beta
   )}
\end{multline}
where $Re(\tau)>1,0< Re(\mu)<1$.
\end{example}
\begin{example}
Here we use equation (\ref{eq:mitt_1}) and set $\lambda \to 0,\mu \to 0,\nu \to 1,\theta \to s,x\to t,m\to m-1,f\to j$ and simplify;
\begin{equation}\label{eq:mitt_3}
\int_0^{\infty } e^{-s t^{\tau }} t^{m-1} E_{\alpha ,\beta }\left(t^{\gamma } \delta \right) \,
   dt=\frac{1}{\tau  s^{m/\tau }}\sum _{j=0}^{\infty } \frac{\Gamma \left(\frac{m+j
   \gamma }{\tau }\right)}{\Gamma (j \alpha +\beta )}\left(\frac{\delta }{s^{\gamma /\tau }}\right)^j 
\end{equation}
where $Re(\tau)>1$.
\end{example}
\begin{example}
Here we use equation (\ref{eq:mitt_1}) and set $\lambda \to 0,\mu \to 0,\nu \to 1,\tau \to 1,\theta \to s,x\to t,m\to m-1,f\to j$ and simplify;
\begin{equation}
\int_0^{\infty } e^{-s t} t^{m-1} E_{\alpha ,\beta }\left(t^{\gamma } \delta \right) \, dt=\frac{1}{s^m}\sum
   _{j=0}^{\infty }\left(\frac{\delta }{s^{\gamma }}\right)^j  \frac{\Gamma (m+j \gamma )}{\Gamma (\beta+j \alpha)}
\end{equation}
where $Re(s)>0$.
\end{example}
\begin{example}
Derivation of equation (C.4(16)) on page 312 in \cite{gorenflo} in terms of the Fox-Wright function \cite{fox}. Errata. Here we use equation (\ref{eq:mitt_1}) and set $\lambda \to 0,\mu \to 0,\nu \to 1$ then $\tau \to 1,\theta \to s,x\to t$ and $\delta \to a,m\to p-1,f\to k$ and simplify;
\begin{multline}\label{eq:mitt_3}
\int_0^{\infty } e^{-s t} t^{-1+p} E_{\alpha ,\beta }\left(a t^{\gamma }\right) \, dt=\sum _{n=0}^{\infty }
   \frac{a^n s^{-p-n \gamma } \Gamma (p+n \gamma )}{\Gamma (n \alpha +\beta )}\neq\frac{1}{s^p}\sum _{n=0}^{\infty }
   \frac{\Gamma (p+\gamma  n) \left(\frac{a}{s^{\alpha }}\right)^n}{\Gamma (\beta +\alpha  n)}
\end{multline}
where $Re(s)>0,Re(\gamma)>1$.
\end{example}
\begin{example}
Derivation of equation (C.4(14)) on page 312 in \cite{gorenflo}.  Here we use equation (\ref{eq:mitt_3}) and set $\gamma \to \alpha ,p\to \beta $ and simplify;
\begin{equation}
\int_0^{\infty } e^{-s t} t^{-1+\beta } E_{\alpha ,\beta }\left(a t^{\alpha }\right) \, dt=\frac{s^{\alpha
   -\beta }}{-a+s^{\alpha }}
\end{equation}
where $Re(s)>1,Re(\alpha)>0$.
\end{example}
\begin{example}
Definite integral of the Mittag-Leffler function in terms of the Hurwitz-Lerch transcendent. Here we use equation (\ref{eq:mitt}) and set $a\to 1,b\to 1,\lambda \to -1$ and simplify using equation [DLMF, \href{https://dlmf.nist.gov/25.14.i}{25.14.1}];
\begin{multline}\label{eq:mitt_lerch}
\int_0^1 \frac{e^{-x^{\tau } \theta } x^m \log ^k\left(\frac{1}{x}\right) E_{\alpha ,\beta }\left(x^{\gamma }
   \delta \right)}{1+x^{\mu } \nu } \, dx\\
=\frac{\Gamma (1+k) }{\mu ^{k+1}}\sum _{j=0}^{\infty } \sum _{l=0}^{\infty }
   \frac{\delta ^l (-\theta )^j }{\Gamma
   (j+1) \Gamma (l \alpha +\beta )}\Phi \left(-\nu ,1+k,\frac{1+m+l \gamma +j \tau }{\mu }\right)
\end{multline}
where $Re(\tau)>0,Re(\mu)>1$.
\end{example}
\begin{example}
Here we use equation (\ref{eq:mitt_lerch}) and form a second equation by replacing $m\to s$ and take their difference. Next apply l'Hopital's rule as $k\to -1$ and simplify;
\begin{multline}
\int_0^1 \frac{e^{-x^{\tau } \theta } \left(x^m-x^s\right) E_{\alpha ,\beta }\left(x^{\gamma } \delta
   \right)}{\left(1+x^{2 \mu }\right) \log \left(\frac{1}{x}\right)} \, dx\\
=\sum _{j=0}^{\infty } \sum _{l=0}^{\infty} \frac{\delta ^l (-\theta )^j }{\Gamma (1+j) \Gamma (l \alpha+\beta )}\log \left(\frac{\Gamma \left(\frac{1+m+l \gamma +j \tau }{4 \mu }\right) \Gamma\left(\frac{1+s+l \gamma +2 \mu +j \tau }{4 \mu }\right)}{\Gamma \left(\frac{1+s+l \gamma +j \tau }{4 \mu}\right) \Gamma \left(\frac{1+m+l \gamma +2 \mu +j \tau }{4 \mu }\right)}\right)
\end{multline}
where $Re(\tau)>1,Re(\mu)>0$.
\end{example}
\begin{example}
Malmsten definite integral form. Here we use equation (\ref{eq:mitt_lerch}) and set $\nu\to 1$ and apply l'Hopital's rule to the first partial derivative with respect to $k$ as $k\to 0$ and simplify in terms of Euler's constant, the Stieltjes constant and digamma function respectively using equations [Wolfram, \href{http://functions.wolfram.com/02.06.02.0001.01}{01}],  [Wolfram, \href{http://functions.wolfram.com/10.05.02.0001.01}{01}], and [Wolfram, \href{http://functions.wolfram.com/06.14.02.0001.01}{01}];
\begin{multline}
\int_0^1 \frac{e^{-x^{\tau } \theta } x^m \log \left(\log \left(\frac{1}{x}\right)\right) E_{\alpha ,\beta
   }\left(x^{\gamma } \delta \right)}{1+x^{\mu }} \, dx\\
=\sum _{j=0}^{\infty } \sum _{l=0}^{\infty } \frac{\delta ^l(-\theta )^j }{2 \mu \Gamma (1+j) \Gamma (l \alpha +\beta )}\left((\gamma +\log (2 \mu )) \left(\psi ^{(0)}\left(\frac{1+m+l \gamma +j \tau }{2 \mu
   }\right)\right.\right. \\ \left.\left.
-\psi ^{(0)}\left(\frac{1+m+l \gamma +\mu +j \tau }{2 \mu }\right)\right)-\gamma_1\left(\frac{1+m+l\gamma +j \tau }{2 \mu }\right)+\gamma _1\left(\frac{1+m+l \gamma +\mu +j \tau }{2 \mu }\right)\right)
\end{multline}
where $Re(\tau)>1,Re(\mu)>0$.
\end{example}
\begin{example}
Here we use equation (\ref{eq:mitt_lerch}) and set $\nu \to 1$ and simplify in terms of the Hurwitz zeta function using [Wolfram, \href{http://functions.wolfram.com/10.06.03.0025.01}{01}];
\begin{multline}\label{eq:mitt_hurwitz}
\int_0^1 \frac{e^{-x^{\tau } \theta } x^m \log ^k\left(\frac{1}{x}\right) E_{\alpha ,\beta }\left(x^{\gamma }
   \delta \right)}{1-x^{\mu }} \, dx=\sum _{j=0}^{\infty } \sum _{l=0}^{\infty } \frac{\delta ^l (-\theta )^j \mu
   ^{-1-k} \Gamma (1+k) \zeta \left(1+k,\frac{1+m+l \gamma +j \tau }{\mu }\right)}{\Gamma (1+j) \Gamma (l \alpha
   +\beta )}
\end{multline}
where $Re(\tau)>1,Re(\mu)>0$.
\end{example}
\begin{example}
Here we use equation (\ref{eq:mitt_hurwitz}) and form a second equation by setting $m\to s$ and take their difference and apply l'Hopital's rule as $k\to 0$ and simplify using [Wolfram, \href{http://functions.wolfram.com/06.14.02.0001.01}{01}];
\begin{multline}
\int_0^1 \frac{e^{-x^{\tau } \theta } \left(x^m-x^s\right) E_{\alpha ,\beta }\left(x^{\gamma } \delta
   \right)}{-1+x^{\mu }} \, dx\\
=\sum _{j=0}^{\infty } \sum _{l=0}^{\infty } \frac{\delta ^l (-\theta )^j \left(\psi
   ^{(0)}\left(\frac{1+m+l \gamma +j \tau }{\mu }\right)-\psi ^{(0)}\left(\frac{1+s+l \gamma +j \tau }{\mu
   }\right)\right)}{\mu  \Gamma (1+j) \Gamma (l \alpha +\beta )}
\end{multline}
where $Re(\tau)>1,Re(\mu)>0$.
\end{example}
\begin{example}
Here we use equation (\ref{eq:mitt_lerch}) and set $\nu\to 1$ and simplify in terms of the Hurwitz zeta function using equation [Wolfram, \href{http://functions.wolfram.com/10.06.03.0073.01}{01}], then take the first partial derivative with respect to $k$ and set $k\to -1/2$ and simplify;
\begin{multline}
\int_0^1 \frac{e^{-x^{\tau } \theta } \left(x^m-x^s\right) \log \left(\log \left(\frac{1}{x}\right)\right)
   E_{\alpha ,\beta }\left(x^{\gamma } \delta \right)}{\left(-1+x^{\mu }\right) \sqrt{\log
   \left(\frac{1}{x}\right)}} \, dx\\
=\sum _{j=0}^{\infty } \sum _{l=0}^{\infty } \frac{\sqrt{\pi } \delta ^l (-\theta
   )^j }{\sqrt{\mu
   } \Gamma (1+j) \Gamma (l \alpha +\beta )}\left(\left(\zeta \left(\frac{1}{2},\frac{1+m+l \gamma +j \tau }{\mu }\right)\right.\right. \\ \left.\left.
-\zeta
   \left(\frac{1}{2},\frac{1+s+l \gamma +j \tau }{\mu }\right)\right) (\gamma +\log (4)+\log (\mu
   ))\right. \\ \left.
-\zeta'\left(\frac{1}{2},\frac{1+m+l \gamma +j \tau }{\mu
   }\right)+\zeta'\left(\frac{1}{2},\frac{1+s+l \gamma +j \tau }{\mu }\right)\right)
\end{multline}
where $Re(\tau)>1,Re(\mu)>0$.
\end{example}
\begin{example}
Malmsten integral form. Here we use equation (\ref{eq:mitt_lerch}) and set $\nu\to 1$ and simplify in terms of the Hurwitz zeta function using equation [Wolfram, \href{http://functions.wolfram.com/10.06.03.0073.01}{01}], then form a second equation by replacing $m\to s$ then take the first partial derivative with respect to $k$ and apply l'Hopital's rule as $k\to -1$ and simplify using [Wolfram, \href{http://functions.wolfram.com/06.11.02.0001.01}{01}];
\begin{multline}
\int_0^1 \frac{e^{-x^{\tau } \theta } \left(x^m-x^s\right) \log \left(\log \left(\frac{1}{x}\right)\right)
   E_{\alpha ,\beta }\left(x^{\gamma } \delta \right)}{\left(1+x^{\mu }\right) \log \left(\frac{1}{x}\right)} \,
   dx\\
=\sum _{j=0}^{\infty } \sum _{l=0}^{\infty } \frac{\delta ^l (-\theta )^j }{2 \Gamma
   (1+j) \Gamma (l \alpha +\beta )}\left(2 \log (2 \mu  \exp (\gamma ))
   \log \left(\frac{\Gamma \left(\frac{1+s+l \gamma +j \tau }{2 \mu }\right) \Gamma \left(\frac{1+m+l \gamma +\mu +j
   \tau }{2 \mu }\right)}{\Gamma \left(\frac{1+m+l \gamma +j \tau }{2 \mu }\right) \Gamma \left(\frac{1+s+l \gamma
   +\mu +j \tau }{2 \mu }\right)}\right)\right. \\ \left.
+\zeta''\left(0,\frac{1+m+l \gamma +j \tau }{2 \mu
   }\right)-\zeta''\left(0,\frac{1+s+l \gamma +j \tau }{2 \mu
   }\right)\right. \\ \left.
-\zeta''\left(0,\frac{1+m+l \gamma +\mu +j \tau }{2 \mu
   }\right)+\zeta''\left(0,\frac{1+s+l \gamma +\mu +j \tau }{2 \mu }\right)\right)
\end{multline}
where $Re(\tau)>1,Re(\mu)>0$.
\end{example}
\begin{example}
Here we use equation (\ref{eq:mitt}) and set $k\to -1,\tau \to 0,\theta \to 1,\lambda \to -\frac{1}{2},\mu \to 2,b\to 1$ next form a second equation by replacing $a\to 1/a$ and take their difference and simplify;
\begin{multline}
\int_0^1 \frac{x^m E_{\alpha ,\beta }\left(x^{\gamma } \delta \right)}{\sqrt{1+x^2 \nu } \left(\log
   ^2(a)-\log ^2(x)\right)} \, dx\\
=\sum _{f=0}^{\infty } \sum _{h=0}^{\infty } \frac{\delta ^f \nu ^h
   \binom{-\frac{1}{2}}{h} }{2 a \Gamma
   (f \alpha +\beta ) \log (a)}\left(-a^{-2 h-m-f \gamma } \Gamma \left(0,(1+2 h+m+f \gamma ) \log
   \left(\frac{1}{a}\right)\right)\right. \\ \left.+a^{2+2 h+m+f \gamma } \Gamma (0,(1+2 h+m+f \gamma ) \log (a))\right)
\end{multline}
where $Re(\delta)>0,Re(\gamma)>1$.
\end{example}
\begin{example}
Here we use equation (\ref{eq:mitt}) and set $\tau \to 0,\theta \to 1,k\to 0,a\to 1$ and simplify;
\begin{multline}
\int_0^b x^{m-1} \left(1-x^{\mu } \nu \right)^{\lambda } E_{\alpha ,\beta }\left(x^{\gamma } \delta \right)
   \, dx\\
=\sum _{f=0}^{\infty } \frac{b^{m+f \gamma } \delta ^f}{(m+f \gamma ) \Gamma (f \alpha +\beta )} \, _2F_1\left(-\lambda ,\frac{m+f \gamma }{\mu
   };\frac{m+f \gamma +\mu }{\mu };b^{\mu } \nu \right)
\end{multline}
where $Re(\gamma)>1,Re(\mu)>0$.
\end{example}
\begin{example}
Here we use equation (\ref{eq:mitt}) and set $\tau \to 0,\theta \to 1,b\to 1,a\to 1,\lambda \to -1$ then take the $n$-th partial derivative with respect to $m$ and simplify;
\begin{multline}\label{eq:mitt_lerch_sum}
\int_0^1 \frac{x^m \log ^k\left(\frac{1}{x}\right) \log ^n(x) E_{\alpha ,\beta }\left(x^{\gamma } \delta
   \right)}{1+x^{\mu } \nu } \, dx\\
=-\sum _{f=0}^{\infty } \frac{\pi  \delta ^f \mu ^{-1-k-n} \csc (k \pi ) \Phi
   \left(-\nu ,1+k+n,\frac{1+m+f \gamma }{\mu }\right)}{\Gamma (-k-n) \Gamma (f \alpha +\beta )}
\end{multline}
where $Re(\gamma)>1,Re(\mu)>0,n=0,1,2,..,$.
\end{example}
\begin{example}
Here we use equation (\ref{eq:mitt_lerch_sum}) and take the first partial derivative with respect to $k$ set $k\to -1/2$ and simplify;
\begin{multline}\label{eq:mitt_lerch_sum1}
\int_0^1 \frac{x^m \log ^n(x) \log \left(\log \left(\frac{1}{x}\right)\right) E_{\alpha ,\beta
   }\left(x^{\gamma } \delta \right)}{\left(1+x^{\mu } \nu \right) \sqrt{\log \left(\frac{1}{x}\right)}} \, dx\\
=\sum
   _{f=0}^{\infty } -\frac{\pi  \delta ^f \mu ^{-\frac{1}{2}-n} }{\Gamma \left(\frac{1}{2}-n\right) \Gamma (f \alpha +\beta )}\left(\Phi \left(-\nu ,\frac{1}{2}+n,\frac{1+m+f
   \gamma }{\mu }\right) \left(\log (\mu )-\psi^{(0)}\left(\frac{1}{2}-n\right)\right)\right. \\ \left.
-\Phi'\left(-\nu ,\frac{1}{2}+n,\frac{1+m+f\gamma }{\mu }\right)\right)
\end{multline}
where $Re(\gamma)>1,Re(\mu)>0,n=0,1,2,..,$
\end{example}
\begin{example}
Here we use equation (\ref{eq:mitt_lerch_sum1}) and form a second equation by replacing $m\to s$ and take their difference and set $n\to -1$ and simplify;
\begin{multline}
\int_0^1 \frac{\left(x^s-x^m\right) \log ^k\left(\frac{1}{x}\right) E_{\alpha ,\beta }\left(x^{\gamma }
   \delta \right)}{\left(1+x^{\mu } \nu \right) \log (x)} \, dx\\
=\sum _{f=0}^{\infty } \frac{\pi  \delta ^f \mu
   ^{-k} \csc (k \pi ) \left(\Phi \left(-\nu ,k,\frac{1+m+f \gamma }{\mu }\right)-\Phi \left(-\nu ,k,\frac{1+s+f
   \gamma }{\mu }\right)\right)}{\Gamma (1-k) \Gamma (f \alpha +\beta )}
\end{multline}
where $Re(\gamma)>1,Re(\mu)>0$
\end{example}
%
%
%
%
\newpage
\section{Summary Table of Double Integrals Involving the Mittag-Leffler function $E_b(-u)$}
Some mathematical constants used in this table are Glaisher's constant, $A$ given by equation (5.17.6) in \cite{dlmf}, Euler's constant, $\gamma$ given by equation (5.2.3) in \cite{dlmf}, Catalan's constant given by equation (25.1.40) in \cite{dlmf} and the digamma function$\psi^{(0)}(z)$ given by equation (5.15.1) in \cite{dlmf}.
\begin{center}
\setlength{\tabcolsep}{16pt} 
\renewcommand{\arraystretch}{2.0} 
\begin{tabular}{ |c|c|c|c| } 
\hline
$f(u,x)$ & $\int_{{\mathbb{R}^{2}_{+}}}f(u,x)dudx$  \\
\hline
$\frac{e^{-x} E_{\frac{1}{4}}(-u)}{\sqrt{u} \sqrt[8]{x} \log
   \left(\frac{i u}{\sqrt[4]{x}}\right)}$ & $\frac{1}{2} i \left(\psi
   ^{(0)}\left(\frac{3}{8}\right)-\psi
   ^{(0)}\left(\frac{7}{8}\right)\right)$  \\ 
   $\frac{e^{-x} E_{\beta }(-u) x^{-(\beta  (m+n))} \left(u^n x^{\beta 
   m}-u^m x^{\beta  n}\right)}{u \log \left(u x^{-\beta }\right)}$ & $\log
   \left(\cot \left(\frac{\pi  m}{2}\right) \tan \left(\frac{\pi 
   n}{2}\right)\right)$  \\
   $e^{-x} u^{m-1} x^{-b m} E_b(-u) \log ^k\left(a u x^{-b}\right)$ & $(2 i
   \pi )^{k+1} \left(-e^{i \pi  m}\right) \Phi \left(e^{2 i m \pi
   },-k,-\frac{i (\log (a)+i \pi )}{2 \pi }\right)$  \\
$e^{-x} u^{m-1} x^{-b m} E_b(-u) \log ^k\left(-u x^{-b}\right)$ & $(2 i
   \pi )^{k+1} \left(-e^{-i \pi  m}\right) \text{Li}_{-k}\left(e^{2 i m \pi
   }\right)$  \\ 
$\frac{e^{-x} x^{-b/2} E_b(-u)}{\sqrt{u} \log ^2\left(-u
   x^{-b}\right)}$ & $-\frac{\pi }{24}$  \\
$\frac{e^{-x} x^{-b/4} E_b(-u)}{u^{3/4} \log ^2\left(-u
   x^{-b}\right)}$ & $-\frac{\sqrt[4]{-1} \left(\pi ^2-48 i C\right)}{96 \pi
   }$  \\
   $\frac{x^{-b/2} e^{-c x} E_b(-u) \log \left(\log \left(a u
   x^{-b}\right)\right)}{\sqrt{u}}$ & $\frac{1}{2} \pi  c^{\frac{b}{2}-1} \left(\log
   \left(\frac{16 \pi ^2 \Gamma \left(-\frac{i (\log (a)+b \log (c)+3 i \pi )}{4
   \pi }\right)^4}{\Gamma \left(-\frac{i (\log (a)+b \log (c)+i \pi )}{4 \pi
   }\right)^4}\right)+i \pi \right)$  \\
   $\frac{e^{-\frac{\pi  x}{4}} E_{\frac{4}{3}}(-u) \log \left(\log
   \left(-\frac{u}{x^{4/3}}\right)\right)}{\sqrt{u} x^{2/3}}$ & $\frac{\pi ^{2/3}}{\sqrt[3]{2}}
   \left(i \pi +\log \left(\frac{16 \pi ^2 \Gamma \left(1+\frac{i \log
   \left(\frac{4}{\pi }\right)}{3 \pi }\right)^4}{\Gamma
   \left(\frac{1}{2}+\frac{i \log \left(\frac{4}{\pi }\right)}{3 \pi
   }\right)^4}\right)\right)$  \\
    $\frac{e^{-x} x^{-b/2} E_b(-u) \log ^k\left(-u x^{-b}\right)}{\sqrt{u}}$ & $\left(1-2^{k+1}\right) (2 i \pi )^{k+1} \zeta (-k)$  \\
     $\frac{e^{-x} x^{-b/2} E_b(-u) \log \left(\log \left(-u
   x^{-b}\right)\right)}{\sqrt{u}}$ & $\pi  \left(\log (4)+\frac{i \pi
   }{2}\right)$  \\
      $\frac{e^{-x} x^{-b/2} E_b(-u) \log \left(\log \left(-u
   x^{-b}\right)\right)}{\sqrt{u} \log \left(-u x^{-b}\right)}$ & $\frac{1}{2} \log
   (2) \left(2 i \gamma +\pi -i \log \left(8 \pi ^2\right)\right)$  \\
       $\frac{e^{-x} x^{-b/2} E_b(-u) \log \left(\log \left(-u
   x^{-b}\right)\right)}{\sqrt{u} \log ^2\left(-u x^{-b}\right)}$ & $\frac{1}{24}
   \pi  \left(-12 \log (A)+\gamma -\frac{i \pi }{2}+\log (2)\right)$  \\
\hline
\end{tabular}
\end{center}
\subsection{Mittag-Leffler function  $E_b(-u)$}
The Mittag-Leffler function is an entire function and is defined in equation (10.46.3) in \cite{dlmf}. A plot of the Mittag-Leffler function over the range defined in this work.
%
\begin{figure}[H]
\includegraphics[scale=0.6]{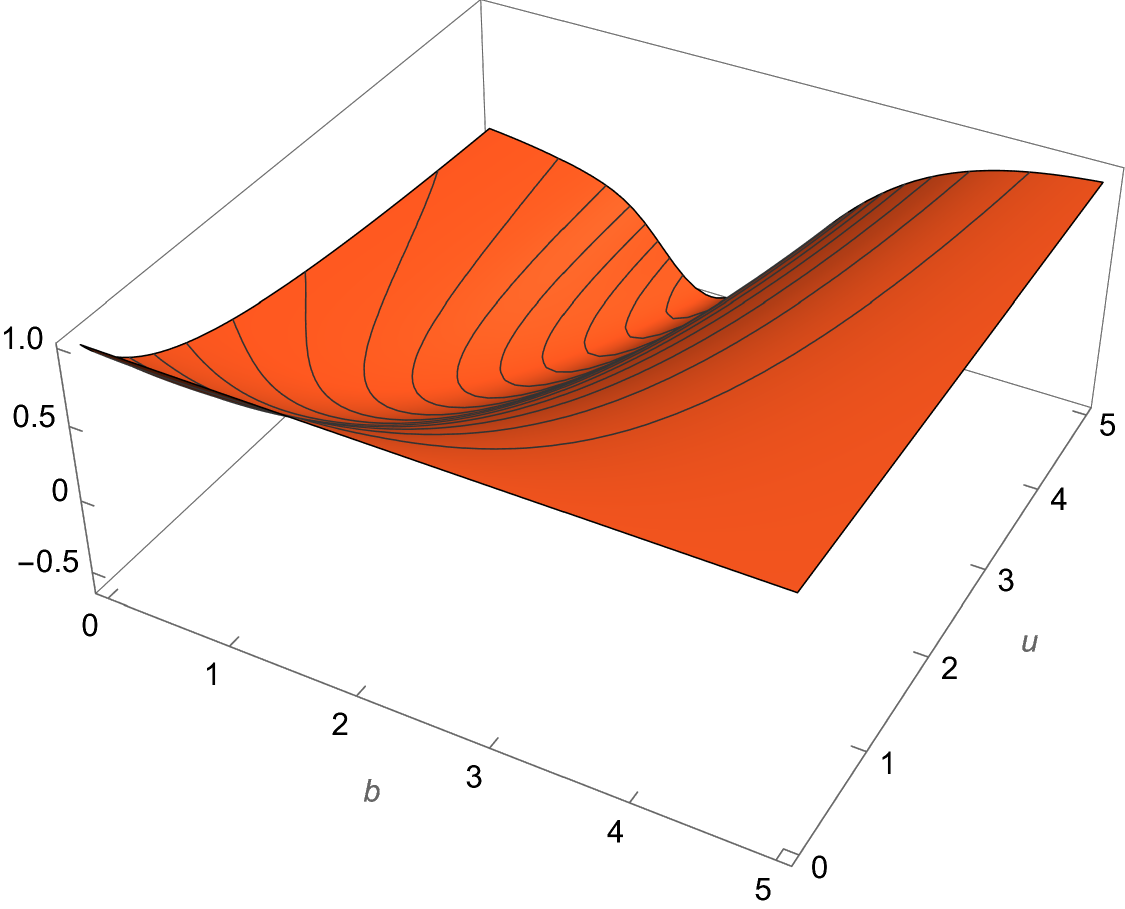}
\caption{ Plot of the Mittag-Leffler function $E_b(-u)$}
   \label{fig:fig2}
\end{figure}
\vspace{-6pt}
\section{Discussion}
In this paper, we have presented a method for deriving a new double integral transform involving the Mittag-Leffler function $E_b(-u)$ along with some interesting definite integrals as special cases using contour integration. The results presented were numerically verified for both real and imaginary and complex values of the parameters in the integrals using Mathematica by Wolfram. In a paper of this type, special attention has be taken to insure accuracy of the formulas.
\section{Funding} This research has no Governement funding.
\section{Conflict of Interests}
The author declares this work has no conflicts of interest.
\section{Competing interests}
The author declares there are no Competing interests with this work.
\section{Data Availability}
No external data was used in this work.
\end{document}